\documentclass[10pt]{article}

\usepackage[utf8x]{inputenc}
\usepackage[english]{babel}
\usepackage[T1]{fontenc}
\usepackage{amsmath,amscd}
\usepackage{amsfonts}
\usepackage{amssymb}
\usepackage{amsthm}
\usepackage{mathrsfs}
\usepackage{graphicx}
\usepackage[all,cmtip]{xy}
\usepackage{textcomp}
\usepackage{url}
\usepackage{amsopn,latexsym,amscd}
\usepackage{color}
\usepackage{todonotes}
\usepackage{fullpage}
\usepackage{amsmath,graphicx}

\newtheorem{teo}{Theorem}[section]

\newtheorem{rem}{Remark}[section]

\newtheorem{cor}{Corollary}[section]
\newtheorem{es}{Example}[section]
\newtheorem{defin}{Definition}[section]

\newcommand{\Z}{{\mathbb{Z}}}

\newcommand{\N}{{\mathbb{N}}}

\newcommand{\Pc}{{\mathcal{P}}}

\newcommand{\Sc}{{\mathcal{S}}}

\newcommand{\Oc}{{\mathcal{O}}}
\newcommand{\Tc}{{\mathcal{T}}}
\newcommand{\Ic}{{\mathcal{I}}}

\title{Nested sets, set partitions and  Kirkman-Cayley dissection numbers}
\author{Giovanni Gaiffi}
\date{\today}

\begin{document}
\maketitle
\begin{abstract}
In this paper   we show a a proof   by explicit bijections of  the famous Kirkman-Cayley formula for the number of dissections of  a convex polygon.
Our  starting point is the bijective correspondence  between  the set of nested sets made by \(k\) subsets of  \(\{1,2,...,n\}\)  with   cardinality   \(\geq 2\)  and the set of partitions of \(\{1,2,...,n+k-1\}\) into \(k\) parts with   cardinality  \(\geq 2\). 
A   bijection  between these two sets can be obtained from  P\'eter L Erd\H{o}s and L.A Sz\'ekely result in \cite{PeErdos};  to make this paper self contained we describe another  explicit bijection  that is  a  variant of their  bijection.

\end{abstract}

\section{Introduction}
Let  \(D_{n+1,k-1}\)  be the number of    dissections of a convex polygon with \(n+1\)  labelled edges by  \(k-1\) diagonals,  such that no two of the diagonals  intersect in their interior.  The  formula for \(D_{n+1,k-1}\)  dates back to     Kirkman  (see \cite{Kirkman}) and Cayley (who gave the first complete proof in \cite{Cayley}):
\[ D_{n+1,k-1}= \frac{1}{k}\binom{n-2}{k-1}\binom{n+k-1}{k-1}\]

In \cite{sikora} and \cite{stanleydissections} one can find two elegant proofs of this formula via bijections.
In this paper  we show another proof  that consists in the description of  two explicit bijections and is based on the combinatorics of nested sets and set partitions. We also provide  a  proof of the classical formula (see Chapter 13 of \cite{gouldenjackson}) for the number of the dissections such that the types  of the internal polygons are prescribed.  

We start a more detailed   outline   of the content of the present paper  by recalling   the notion of nested set.
This notion  appeared in geometry  in connection with models of configuration spaces,    first in Fulton and MacPherson paper \cite{fultonmacpherson}, then, with more generality, in De Concini and Procesi's papers    \cite{DCP2}, \cite{DCP1} (and later in \cite{DCPkirwan}) on wonderful models of subspace arrangements. 

Some  generalizations of De Concini and Procesi's definition  successively appeared in various combinatorial contexts. We refer the reader  to \cite{feichtnerkozlovincidence} where the case of  meet semilattices is dealt with, or to \cite{postnikov} and 
\cite{postnikoreinewilli} where some  polytopes named nestohedra are constructed, and finally to   \cite{petric2} where one can find a comparison among various definitions in the literature.

Let us denote by \(\Pc_2(\{1,2,...,n\})\) the subset of the power set \(\Pc(\{1,2,...,n\})\) whose elements have cardinality greater than or equal to 2. 
The following  definition  of nested set  of   \(\Pc_2(\{1,2,...,n\})\)  is a special case    of  the  more general  combinatorial definition. 
\begin{defin}
\label{def:nestedpowerset}
Let \(n\geq 2\). A  subset  \(S\) of \(\Pc_2(\{1,2,...,n\})\)  is a nested set if and only if it contains \(\{1,2,...,n\}\) and for any \(I,J\in S\) we have that either \(I\subset J\) or \(J\subset I\) or \(I\cap J=\emptyset\).
We will denote by \(\Sc_2(n,k)\) the set of the nested sets \(S\)  of  \(\Pc_2(\{1,2,...,n\})\) such that \(|S|=k\). 

\end{defin}

Now we observe that there is a bijective correspondence    between \(\Sc_2(n,k)\) and the set \(\Tc_2(n+k-1,k)\) of partitions of the set \(\{1,2,...,n+k-1\}\) into \(k\) parts of cardinality greater than or equal to 2.
A   bijection  between these two sets  can be obtained as a particular case of the bijection between   rooted trees on \(n\) leaves and partitions proven by  P\'eter L Erd\H{o}s and L.A Sz\'ekely (see Theorem 1 of     \cite{PeErdos}).

To make our paper self contained we show in Section \ref{sec:maintheorem} (Theorem \ref{bijectionnested}) an explicit bijection  between \(\Sc_2(n,k)\) and \(\Tc_2(n+k-1,k)\) that is a variant of the bijection obtained from 
 \cite{PeErdos}.

We notice that, for fixed \(n\),  the cardinalities \(|\Sc_2(n,k)|=|\Tc_2(n+k-1,k)|\)  are the Ward numbers (see \cite{ward}, and the sequence A134991 of OEIS) and can be read along the  diagonals in the table of the 2-associated Stirling numbers of the second kind  at  page 222 of \cite{comtet}. They can as well be interpreted  as the face numbers in the tropical Grassmannian \(G(2,n+1)\), i.e. the space of phylogenetic trees \(T_{n+1}\) (see \cite{billeravogtmann},\cite{feichtnercomplexestrees},\cite{speyersturmfels}). For a description of generating formulas for these numbers one can  also see Chapter 5 of \cite{stanleyenumerative2} (in particular Section 5.2.5 and Exercise 5.40).

As  it is well known (see Section \ref{sec:final}, Figure \ref{dissection}), the dissections of a convex polygon with \(n+1\) edges  by \(k-1\) diagonals are in bijection with the parenthesizations with \(k\) couples of parentheses of a list of \(n\) distinct numbers \(a_1,a_2,...,a_n\) (the maximal couple of parentheses  is included and  every couple of parentheses contains at least two numbers). 
We will denote   by \(\Sc_2((a_1,a_2,...,a_n),k)\) the set of all these parenthesizations: as a consequence of the remark above,  \(|\Sc_2((a_1,a_2,...,a_n),k)|=D_{n+1,k-1}\).

Now, an `ordered' variant of Theorem \ref{bijectionnested} (stated in  Section \ref{sec:maintheorem}  as   Theorem \ref{orderedbijectionnested})   describes a bijection between \({\displaystyle \bigcup_{\sigma\in S_n}\Sc_2((a_{\sigma(1)},a_{\sigma(2)},...,a_{\sigma(n)}),k)}\)
 and the set of {\em admissible  internally ordered} partitions of \(\{1,2,...,n+k-1\}\) into \(k\) parts, i.e. partitions whose parts  have cardinality \(\geq 2\) and are equipped with an internal total ordering. 

This leads to a proof   (in Section \ref{sec:final}, Corollary \ref{cor:cayleyformula})  of Kirkman-Cayley  formula, since the  enumeration of the admissible  internally ordered partitions is  provided, again by an explicit bijection,  by Theorem \ref{teo:semiordinate} in Section \ref{sec:enumeration}.

Even if our proof is purely combinatorial,  in the end of Section  \ref{sec:final} we  sketch out  a geometric interpretation: our  argument corresponds to counting in two different ways the boundary components of a  spherical model of \({\overline M}_{0,n+1}\),  the moduli space of real stable \(n+1\)-pointed   curves of genus 0.

The classical formula for the number of dissections of a convex polygon with \(n+1\) edges such that the types of the internal polygons are prescribed    also follows, in this combinatorial picture,  as an another quick application of Theorem \ref{orderedbijectionnested} (see Section \ref{sec:final}, Corollary  \ref{prescribedpolygons}).

\section{Nested sets and set partitions}
\label{sec:maintheorem}
For every \(i\in \Z\) and \(k\in \N\) let us denote by \([i,i+k]\) the interval  of integers \(\{i,i+1,...,i+k\}\).

As we pointed out in the Introduction,  from  Theorem 1 in \cite{PeErdos} one obtains a  bijection between \(\Sc_2(n,k)\) and the set \(\Tc_2(n+k-1,k)\) of partitions of \([1, n+k-1]\) into \(k\) parts of cardinality greater than or equal to 2. 
To make our paper self contained, the  first part of this  section is devoted to showing  an explicit    bijection between \(\Sc_2(n,k)\) and  \(\Tc_2(n+k-1,k)\)  that is a variant of the one that can be deduced from \cite{PeErdos}.
In the second part of the present  section an `ordered' version of this bijection, that involves partitions into ordered sets,  is  described (Theorem  \ref{orderedbijectionnested}).

\begin{defin}
\label{ordering}
We fix the following  (strict) partial ordering in \(\Pc_2(\{1,2,...,n\})\): given two sets \(I\) and \(J\) in \(\Pc_2(\{1,2,...,n\})\) we put \(I< J\)  if the minimal element in \(I\) is less than the minimal element in \(J\).
\end{defin}

\begin{teo}
\label{bijectionnested}Let us consider  two integers  \(n,k\) with \(n \geq 2\), \(n-1\geq k\geq 1\). 
There is a bijection between \(\Sc_2(n,k)\) and the set \(\Tc_2(n+k-1,k)\) of partitions of \([1, n+k-1]\) into \(k\) parts of cardinality greater than or equal to 2.
\end{teo}
\begin{proof}
Let us  consider a nested set \(S\in \Sc_2(n,k)\). It can be represented by an oriented  rooted tree on \(n\) leaves labelled by the numbers \(1,2,...,n\) in the following way. We consider the set \({\tilde S}=S\cup\{1\}\cup \{2\}\cup\cdots \cup \{n\}\).
Then the tree coincides  with the Hasse diagram of  \({\tilde S}\) viewed as a poset by the  inclusion relation:  the root is \(\{1,2,...,n\}\)  and the orientation goes from the root to the leaves, that are the vertices \(\{1\},\{2\},\ldots, \{n\} \) labelled respectively by the numbers \(1,2,...,n\). 

We observe that we can partition the set of   vertices of the tree into {\em levels}:  level 0 is made by the leaves, and in general,   level \(j\)  is made by the vertices \(v\) such that the maximal length of an oriented  path that connects \(v\) to a leaf is \(j\).
We notice that, since the nested set has \(k\) elements, there are \(k\) internal vertices of this  tree, including the root.
Now we can label the internal vertices of the tree in the following way. Let us suppose that there are \(q\) vertices in level 1. These  vertices correspond, by the nested property, to pairwise disjoint  elements of  \(\Pc_2(\{1,2,...,n\})\)  and therefore we can totally order them using the ordering of Definition \ref{ordering}. Then  we  label them with  the numbers from \(n+1\) to \(n+q\) (the label  \(n+1\)  goes to the minimum, while \(n+q\) goes to the maximum).

At the same way, if there are   \(t\) vertices in level 2, we can label them with the numbers from \(n+q+1\) to \(n+q+t\) , and so on. At the end of the process, the root is labelled with the number \(n+k\): we have obtained an oriented labelled tree with \(n+k\) vertices, where at least two edges  stem from each internal vertex and  the leaves are labelled by  the numbers from 1 to \(n\).

We can now  associate to such a tree a partition in \(\Tc_2(n+k-1,k)\) by assigning  to every internal vertex \(v\) (including the root) the set of the labels of the vertices covered by \(v\). An example of this process is provided by Figure \ref{labellednestedtree}.

We have therefore described a map \(\phi\: : \: \Sc_2(n,k) \rightarrow \Tc_2(n+k-1,k)\). Let us  now describe its inverse.

 \begin{figure}[h]

 \center
\includegraphics[scale=0.6]{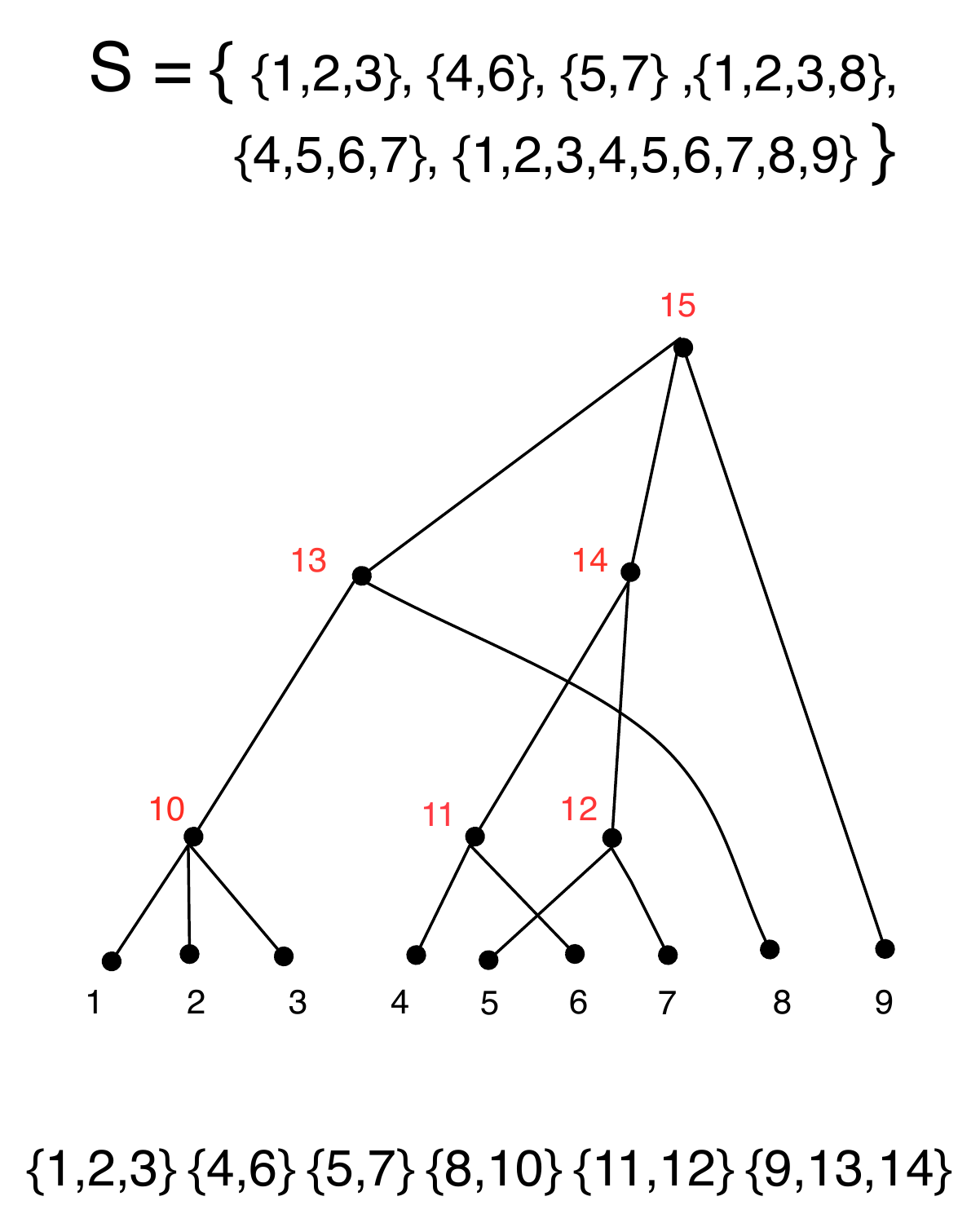}
\caption{A nested set \(S\) with 6 elements  in \(\Pc_2(\{1,2,...,9\})\) (top of the picture),   its associated oriented rooted tree and its associated partition in \(\Tc_2(14,6)\) (bottom of the picture).}
 \label{labellednestedtree}
\end{figure}

Let us consider a partition \(P\) in \(\Tc_2(n+k-1,k)\) with \(k>1\) (the case \(k=1\) is trivial). 
We observe that, since there are \(k\) parts in the partition,  at least one of these parts  is a subset of  \([1,n]\).
Let \({\tilde A_1},{\tilde A_2},...,{\tilde A_i}\) (with \(i\leq k-1\)) be the parts  that are subsets of  \([1,n]\); here  we  have indexed these sets according to the  ordering of Definition \ref{ordering}.

Now, if \(i=k-1\)  the set  \(S=\{\{1,2,...,n\}, {\tilde A_1},{\tilde A_2},...,{\tilde A_i} \}\) is the only one  nested set  in \(\Sc_2(n,k)\) such that  \(\phi(S)=P\). We observe that in the tree associated with this nested set the vertices in  level 1 correspond to the sets \({\tilde A_1},{\tilde A_2},...,{\tilde A_i}\). 

If  \(i<k-1\) we notice  that at least one of the remaining \(k-i\) parts  in  the initial partition  \(P\) is a subset of the complement of \([n+i+1,n+k-1]\) in \([1, n+k-1]\), since the set \([n+i+1,n+k-1]\) has cardinality \(k-i-1\). 

Let us then denote by \(A_{i+1},...,A_{i+s}\) (with \(1\leq s\leq k-1-i\)) these remaining parts  included in  the complement of \([n+i+1,n+k-1]\). For  every \(t=1,...,s\) we associate to  \(A_{i+t}\) 
the set
\[{\tilde A_{i+t}}=\left( A_{i+t}\cap [1,n]\right )\; \cup \; \bigcup_{h \in A_{i+t}\cap [n+1,  n+i]}{\tilde A_{h-n}}\]
The indices have been chosen  in such a way  that, according to  the ordering of Definition \ref{ordering}, we have \({\tilde A_{i+1}}<{\tilde A_{i+2}}<\cdots <{\tilde A_{i+s}}\).

Now, if \(i+s=k-1\) the process stops, and the set \(S=\{\{1,2,...,n\}, {\tilde A_1},{\tilde A_2},...,{\tilde A_i}, {\tilde A_{i+1}},{\tilde A_{i+2}},\cdots ,{\tilde A_{i+s}} \}\) is the only one  nested set  in \(\Sc_2(n,k)\) such that   \(\phi(S)=P\).
We observe that in the tree associated with this nested set the vertices in  level 1 correspond to the sets \({\tilde A_1},{\tilde A_2},...,{\tilde A_i}\) and the vertices in  level 2 correspond to the sets \({\tilde A_{i+1}},{\tilde A_{i+2}},\cdots ,{\tilde A_{i+s}}\).

If \(i+s< k-1\) we can continue and we can construct, level after level,   the only one  nested set  in \(\Sc_2(n,k)\) such that  \(\phi(S)=P\).


\end{proof}

\begin{rem}

We notice that this  bijection introduces on \(\Sc_2(n,k)\) an action of the symmetric group \(S_{n+k-1}\). As one can easily check, when \(k>3\) this action  is not compatible with the usual   \(S_n\) action.  A geometric application of this remark will be discussed in the paper \cite{callegarogaiffi3}.
\end{rem}

We now state a variant of the theorem above, where  nested sets of ordered lists and internally ordered partitions come into play.

\begin{defin}
Let us consider   a  list of distinct numbers \(a_1,a_2,...,a_n\) (with \(n\geq 2\)), and a parenthesization  of  \(a_1,a_2,...,a_n\)  that includes  the maximal couple of parentheses \((a_1,a_2,...,a_n)\) and such that every couple of parentheses contains at least two numbers. We will  call this parenthesization   {\em a nested set of the list 
} \(a_1,a_2,...,a_n\),  
 and we will denote by  \(\Sc_2((a_1,a_2,...,a_n),k)\) the set whose elements are the  nested sets \(S\) of the list \(a_1,a_2,...,a_n\)   with  \(|S|=k\). Furthermore, we will denote by  \(\Oc\Sc_2((a_1,a_2,...,a_n),k)\) the set:
\[ \Oc\Sc_2((a_1,a_2,...,a_n),k)=\bigcup_{\sigma\in S_n}\Sc_2((a_{\sigma(1)},a_{\sigma(2)},...,a_{\sigma(n)}),k)\]
\end{defin}
We observe that if \(\{a_1,a_2,...,a_n\}=\{1,2,..,n\}\) we  can associate to  a nested set of the list 
 \(a_1,a_2,...,a_n\) a nested set in \(\Sc_2(n,k)\) in the following obvious way: we associate to every couple of parentheses the set of the numbers contained in it.

\begin{defin}
\label{def:internallyordered}
An {\em  internally ordered \(k\)-partition} (with \(k\geq 1\)) of  a finite set \(X\) is given by 
\begin{itemize}
\item an unordered partition \(\Pc=\{P_1, P_2,...,P_k\} \) of \(X\) into \(k\) parts;
\item a complete   ordering of the elements of \(P_i\) for every \(1\leq i\leq k\).

\end{itemize} 

If the cardinality of all the subsets \(P_1,P_2,...,P_k\)  is greater than or equal to 2 we say that the internally  ordered partition is {\em admissible}. 
\end{defin}

\begin{teo}
\label{orderedbijectionnested}
Let us consider  two integers  \(n,k\) with \(n \geq 2\), \(n-1\geq k\geq 1\). 
There is a bijection between \(\Oc\Sc_2((1,2...,n),k)\)  and the set \(\Ic\Tc_2(n+k-1,k)\) of  the admissible internally  ordered   \(k\)-partitions of \([1, n+k-1]\).
\end{teo}
\begin{proof}

Let us put \(a_1=1,a_2=2,...,a_n=n\).  Starting  from a nested set \(S\) in  \(\Sc_2((a_{\sigma(1)},a_{\sigma(2)},...,a_{\sigma(n)}),k)\)  we can construct, in a similar way as  in the proof of Theorem \ref{bijectionnested},  an oriented rooted   labelled  tree.

The only difference from the construction described in  the proof of Theorem \ref{bijectionnested}  is the following one:   we draw the leaves  labelled by \(a_{\sigma(1)},a_{\sigma(2)},...,a_{\sigma(n)}\) from left to right, and, since we are dealing with a nested set of the  list \(a_{\sigma(1)},a_{\sigma(2)},...,a_{\sigma(n)}\), in the picture of  the tree   no two of the edges intersect in their interiors.  

Finally we  construct an internally ordered  admissible \(k\)-partition of \([1, n+k-1]\) according to the following rule. 
The parts \(P_1,...,P_k\) of the partition are produced by the internal vertices of the tree. The part  \(P_i\) is obtained from the vertex labelled by \(n+i\) and it is  an ordered set constructed in this way: if, from left to right, the vertices covered by the vertex labelled by \(n+i\) are labelled by \(b_1,b_2,...,b_r\), then \(P_i\) is the set \(\{b_1,b_2,...,b_r\}\) ordered by  \(b_1\prec b_2\prec \cdots \prec b_r\). 

We have described a map \(\Gamma\: : \: \Oc\Sc_2((1,2,...,n),k) \rightarrow \Ic\Tc_2(n+k-1,k)\).
The inverse map is constructed by associating a tree to  an internally ordered  admissible \(k\)-partition \(\{P_1,...,P_k\}\)  of \([1 , n+k-1]\). The idea is essentially the same as in the  proof of Theorem \ref{bijectionnested},
with the only difference that we keep into account the internal orderings of the sets \(P_1,...,P_k\) and we draw the edges from left to right according to this ordering. At the end we produce a nested set of a permutation of the list \(1,2,...,n\) as it is illustrated in Figure \ref{orderednestedtree}.

\begin{figure}[h]

 \center
\includegraphics[scale=0.6]{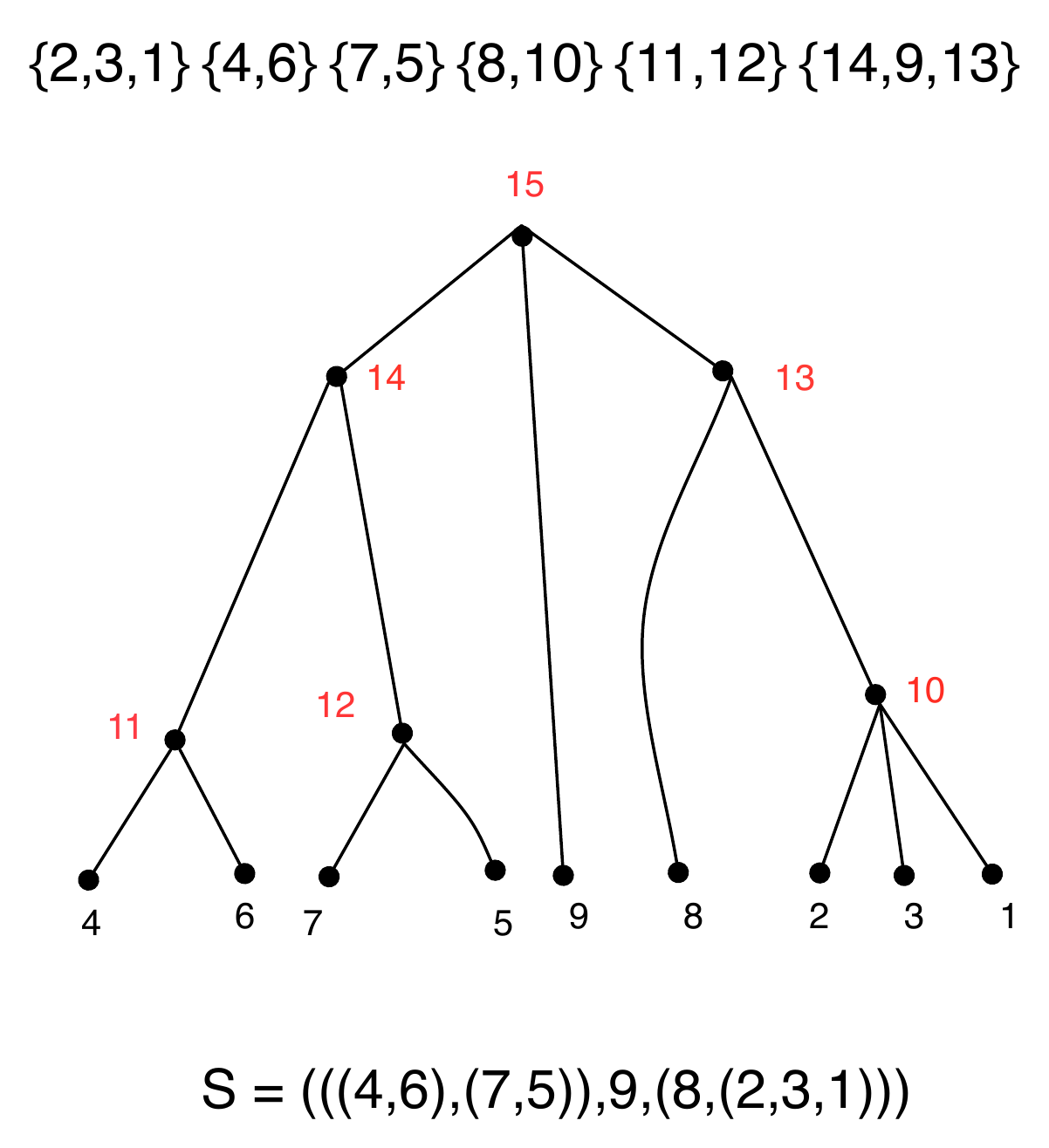}
\caption{The  internally ordered  partition of \([1,14]\) into 6 parts  that is on top of the picture (the internal orderings \(\prec\) are obtained reading from left to right), produces, via  its associated oriented rooted tree, the parenthesization of the list \(4,6,7,5,9,8,2,3,1\) written  on the bottom of the picture.}
 \label{orderednestedtree}
\end{figure}

 \end{proof}

\begin{rem}
A slight modification of the proof of  Theorem \ref{bijectionnested} shows that, as in \cite{PeErdos}, we can extend the bijection in the statement to a  bijection between the set of rooted trees  that have \(k\)  vertices (including the root) and  \(n\) labelled leaves and the set of all the partitions of \([1,n+k-1] \) into \(k\) parts.
 A similar remark applies to Theorem \ref{orderedbijectionnested}.
\end{rem}

\section{Enumeration of internally  ordered \(k\)-partitions}
\label{sec:enumeration}
In the preceding section, Theorem \ref{orderedbijectionnested} has pointed out   an  interesting combinatorial aspect of  the  admissible internally  ordered \(k\)-partitions of \([1, n+k-1]\).
In this section we are going to  count them  by showing an explicit bijection.
 It is useful to add to these partitions a further structure: we  mark    one of the parts, that becomes a {\em distinguished} part.

\begin{defin}

\label{def:semiordered}
A {\em distinguished  internally  ordered \(k\)-partition} (with \(k\geq 1\)) of  a finite set \(X\) is an admissible internally ordered \(k\)-partition where one of the  parts  of the partition is distinguished. 
We will denote by \(X_1\) this distinguished set, therefore the partition is made by  the parts  \(X_1,P_2,...,P_k\).
\end{defin}

\begin{rem}
\label{remordering}
In the definition above, the partition of \(X-X_1\) provided by the sets \(P_i\) is  unordered.
Let  \(p_{i1}\) be, for every \(i\), the smallest  element of \(P_i\) with respect to its  internal ordering.
From now on we will consider the case when \(X\subset \Z\) and we will  assign, by convention,  the indices to the sets \(P_i\) in such a way that  if \(i<j\)   than \(p_{i1}<p_{j1}\).
\end{rem}

\begin{teo}
\label{teo:semiordinate}
Given two integers  \(n,k\) with \(n \geq 2\), \(n-1\geq k\geq 1\),  there is a bijection between the set of distinguished  internally  ordered \(k\)-partitions of \([1,n+k-1]\) and the triples \((I, \sigma, D)\) where 
\begin{itemize}
\item \(I= i_1, i_2, ..., i_n\)  is a sublist of cardinality \(n\) extracted from the list \(L=1,2,..., n+k-1 \), 
\item \(\sigma \)  is a permutation in the symmetric group \(S_n\), 
\item   \(D\) is a sublist of cardinality \(k-1\) extracted from the list  \( i_{\sigma(2)}, ..., i_{\sigma(n-1)}\) (in particular,  for every  \(n\geq 2\), if \(k=1\)  then  \(D\) is the empty list).
\end{itemize}
Therefore the number of distinguished  internally  ordered \(k\)-partitions of \([1,n+k-1]\) is \[n!\binom{n-2}{k-1}\binom{n+k-1}{k-1}\]

\end{teo}
\begin{proof}
Let us   consider a triple \((I, \sigma, D)\) as in the statement of the theorem.

Then \(I= i_1, i_2, ..., i_n\) (with  \(i_1<i_2< \cdots < i_n\)) and we  denote by \(J=j_1,j_2,...,j_{k-1}\) the sublist of \(L\)  made by the numbers that do not belong to \(I\) (notice that \(j_1< j_2< \cdots < j_{k-1}\)).
Now we  use  \(\sigma \) to permute the list \(I\), obtaining a list
\[\sigma I= i_{\sigma(1)}, i_{\sigma(2)}, ..., i_{\sigma(n)}\]
To fix the notation we put    \(D= i_{\sigma(d_1)}, ..., i_{\sigma(d_{k-1})}\)  with \(d_1\geq 2, d_{k-1}\leq n-1\). 


Let us  show how to associate to  \((I, \sigma, D)\) a distinguished  internally  ordered \(k\)-partition of \([1,n+k-1]\). First  we put \(X_1\) to be equal to the set \(\{i_{\sigma(n)}, i_{\sigma(1)},..., i_{\sigma(d_1-1)}\}\) equipped with the ordering \(i_{\sigma(n)}\prec i_{\sigma(1)}\prec \cdots \prec i_{\sigma(d_1-1)}\).
Then we build  the sets \[P_2=\{j_1, i_{\sigma(d_1)}, i_{\sigma(d_1+1)},..., i_{\sigma(d_2-1)}\}\]
\[P_3=\{j_2, i_{\sigma(d_2)}, i_{\sigma(d_2+1)},..., i_{\sigma(d_3-1)}\}\]
\[...\]
\[P_k=\{j_{k-1}, i_{\sigma(d_{k-1})}, i_{\sigma(d_{k-1}+1)},..., i_{\sigma(n-1)}\}.\]
and for every \(2\leq i\leq k\) we order the elements of \(P_i\) as displayed above (for instance, in \(P_2\) we have \(j_1\prec i_{\sigma(d_1)}\prec  i_{\sigma(d_1+1)}\prec\cdots \prec  i_{\sigma(d_2-1)}\)).

We notice that the sets \(P_2,...,P_k\) form  an unordered  partition of \([1,n+k-1]-X_1\), indexed according to Remark \ref{remordering}.  
In conclusion, starting from the triple \((I, \sigma, D)\) we have produced a distinguished  internally  ordered \(k\)-partition of \([1,n+k-1]\) (see also the Example \ref{esempio} at the end of this section). It is easy to check that this  map is injective.

Let us now show that all the distinguished  internally ordered \(k\)-partitions of \([1,n+k-1]\) can be obtained starting from a triple. If \(k=1\) this is trivial. Let then \(k\geq 2\) and let   \(X_1,P_2,...,P_k\) be such a partition. 
More in detail, let \(X_1\) be the ordered set \(\{x_1,x_2,...,x_{|X_1|}\}\), and, for every \(i\) with \(2\leq i\leq k\), let \(P_i\) be the ordered set \(P_i=\{p_{i1},p_{i2},...,p_{i|P_i|}\}\).
 According to the convention described in Remark \ref{remordering},  the indices of the 
\( P_i\)'s  satisfy  \(p_{21}<\cdots < p_{k1}\). Let us denote by \(J\) the list \(p_{21},\ldots , p_{k1}\) and by \(I=i_1,...,i_n\) the sublist of \(L\) that is complementary to \(J\).
Then we choose the permutation \(\sigma \in S_n\) defined as the permutation such that  the list
\[\sigma I= i_{\sigma(1)},...,i_{\sigma(n-1)}, i_{\sigma(n)}\] coincides with the  list 
\[x_2,..,x_{|X_1|}, p_{22},p_{23},...,p_{2|P_2|},p_{32},p_{33},...,p_{3|P_3|},p_{42},...,p_{4|P_4|},...,p_{k,2},..., p_{k|P_k|},x_1\]

As a last step,  we extract from the list \(\sigma I\), after cutting out its  first element and  its  last element,  the sublist of cardinality \(k-1\) \[D=i_{\sigma(|X_1|) }, i_{\sigma(|X_1|+|P_2|-1) },  i_{\sigma(|X_1|+|P_2|+|P_3|-2) }, ..., i_{\sigma(|X_1|+|P_2|+\cdots + |P_{k-1}|-(k-2)) }\]
 (in particular, if \(k=2\) this list is \(D=i_{\sigma(|X_1|) }\)).

 We observe that, by construction, the  triple \((I,\sigma, D)\) is associated with  the distinguished  internally  ordered partition \(X_1,P_2,...,P_k\).

\end{proof}

\begin{es}
\label{esempio}
Let \(n=7\), \(k=4\), and let us consider the triple \((I,\sigma, D)\) where:
\begin{itemize}
\item \({\displaystyle I=\stackrel {i_1}1, \stackrel {i_2}2,\stackrel {i_3}3,\stackrel {i_4}4,\stackrel {i_5}6,\stackrel {i_6}9,\stackrel {i_7}{10}} \) is a sublist of \(L=1,2,3,4,5,6,7,8,9,10\).
\item \(\sigma\)  is the permutation in  \(S_7\) such that   \[\sigma I=\stackrel {i_{\sigma(1)}}2, \stackrel{i_{\sigma(2)}}6,\stackrel {i_{\sigma(3)}}1,\stackrel {i_{\sigma(4)}}{10},\stackrel {i_{\sigma(5)}}3,\stackrel {i_{\sigma(6)}}9,\stackrel {i_{\sigma(7)}}{4} \]
\item \({\displaystyle D=\stackrel {i_{\sigma(3)}}1, \stackrel{i_{\sigma(5)}}3,\stackrel {i_{\sigma(6)}}9}\)  is a sublist of the list \(6,1,10,3,9\) (i.e. the list  \(\sigma I\) without its initial and final term).

\end{itemize}
Let us associate to this triple a distinguished  internally  ordered  partition \(X_1,P_2,P_3,P_4\) of \([1,10]\) according to the bijection of Theorem \ref{teo:semiordinate}.

The set \(X_1\) is \(\{i_{\sigma(7)}, i_{\sigma(1)}, i_{\sigma(2)}\}\),  i.e. \(X_1=\{4,2,6\}\), ordered from left to right: \(4\prec 2 \prec 6\).

Now we notice that the complement of  \(I=1,2,3,4,6,9,10\) in \(L=1,2,3,4,5,6,7,8,9,10\), is the list \(J=5,7,8\).
Then \(P_2=\{5, i_{\sigma(3)}, i_{\sigma(4)}\}\), i.e. \(P_2=\{5,1,10\}\) equipped with the ordering  \(5\prec 1 \prec 10\).  In an analogous way one finds \(P_3=\{7,3\}\) and \(P_4=\{8,9\}\).

\end{es}

\section{Kirkman-Cayley dissection numbers and numbers  of dissections of a   prescribed type}
\label{sec:final}

As it is well known,  \(D_{n+1,k-1}\) coincides with  \(|\Sc_2((1,2,...,n),k)|\): an  explicit bijection between  \(\Sc_2((1,2,...,n),k)\)  and  the set of the dissections of a convex polygon with \(n+1\) labelled edges by \(k-1\) non intersecting diagonals is  illustrated by Figure \ref{dissection}.

 \begin{figure}[h]
 
 \center
\includegraphics[scale=0.4]{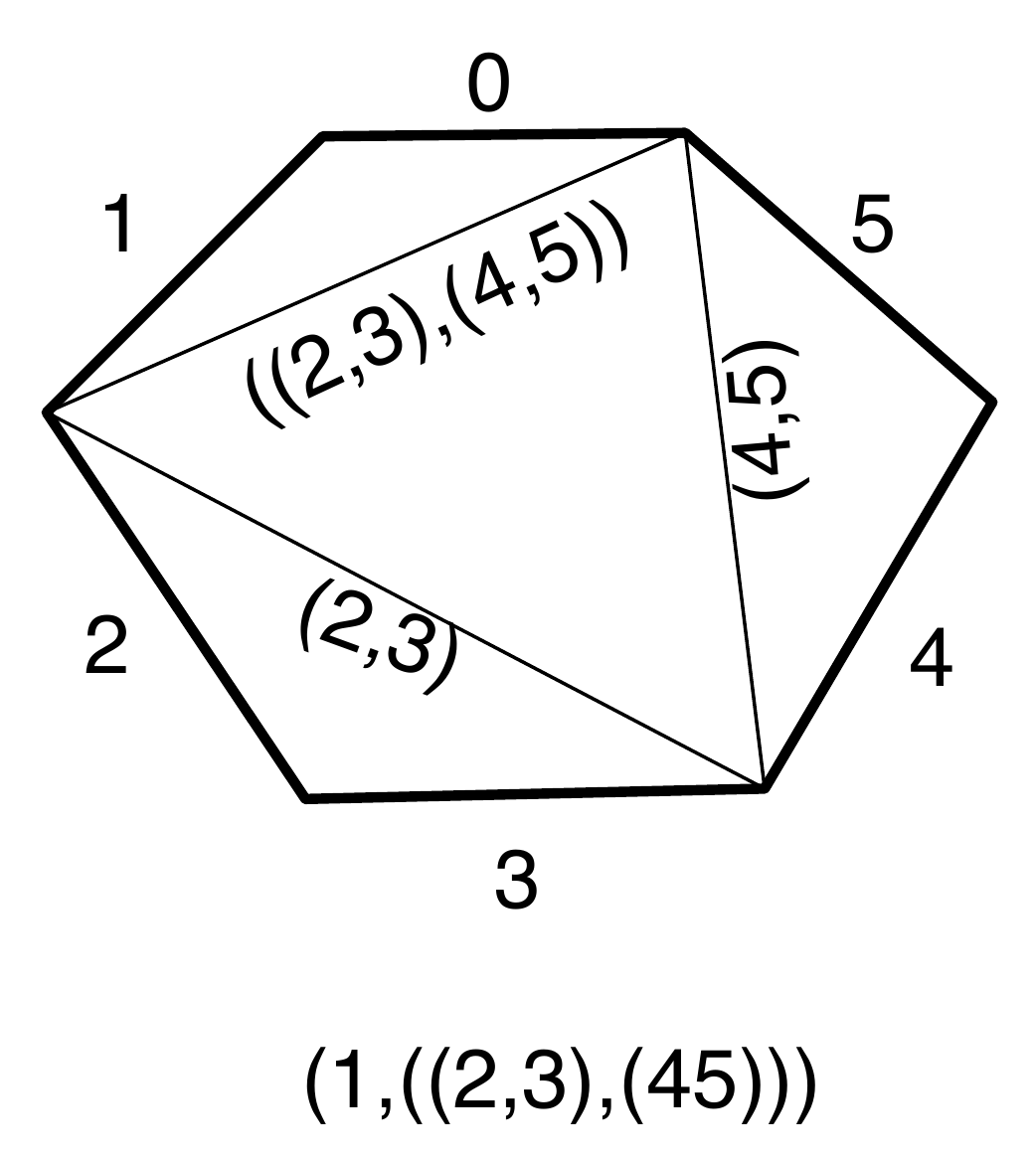}
\caption{This dissection of the hexagon with  3 diagonals produces the  parenthesized list  \((1,((2,3),(4,5)))\) that has 4  couples of parentheses.}
\label{dissection}
\end{figure}

In view of this, a   proof of Kirkman-Cayley  formula via bijections immediately follows as  a corollary of the Theorems  \ref{orderedbijectionnested} and \ref{teo:semiordinate}. 

\begin{cor}[A proof of Kirkman-Cayley formula]
\label{cor:cayleyformula}
\[D_{n+1,k-1}= \frac{1}{k}\binom{n-2}{k-1}\binom{n+k-1}{k-1}\]
\end{cor}
\begin{proof}
We can count the cardinality of \(\Oc\Sc_2((1,2,...,n),k)\) in two different ways. On one hand  it is equal to 
\[n! D_{n+1,k-1}\]
since  there are \(n!\) different lists based on the set of numbers \(\{1,2,...,n\}\)  and for every list we take into account all its nested sets with \(k\) couples of parentheses.

On the other hand, by the Theorem  \ref{orderedbijectionnested}, the cardinality of \(\Oc\Sc_2((1,2,...,n),k)\)  is equal to the number of   admissible internally  ordered  \(k\)-partitions of \([1,n+k-1]\).  This can be obtained dividing by \(k\) the number of distinguished internally  ordered  \(k\)-partitions of \([1,n+k-1]\).

Therefore by Theorem \ref{teo:semiordinate} we have:
\[n! D_{n+1,k-1}= \frac{1}{k}n!\binom{n-2}{k-1}\binom{n+k-1}{k-1}\]
that, after dividing by \(n!\), gives Kirkman-Cayley formula.
\end{proof}
As another application of Theorem  \ref{orderedbijectionnested} we show a proof  of the  formula that counts the number of dissections of a prescribed type. This is a classical result (for another proof see for instance  Chapter 13 of \cite{gouldenjackson}; see also  Section 2.3. of \cite{devadossread}). 

\begin{defin}
Given  a dissection of a convex polygon with \(n+1\)  labelled edges by  \(k-1\) diagonals,  such that no two of the diagonals  intersect in their interior, we say that the dissection is of type \((i_1^{m_1},i_2^{m_2},...,i_s^{m_s})\), with \(3\leq i_1<i_2 <\cdots < i_s\leq n+1\), if    the dissection  is made by  \(m_j\) polygons with \(i_j\) edges,  for every \(j=1,2,...,s\).
\end{defin}

\begin{rem}
\label{propfacile}
As one can immediately check, the  numbers that appear in the above definition  satisfy  the relations  \({\displaystyle \sum_{j=1}^s m_j=k}\)  and \({\displaystyle \sum_{j=1}^s i_j m_j=n+2k-1}\).
\end{rem}



\begin{cor}[A proof of the formula for the dissections of a  prescribed type]
\label{prescribedpolygons}
Let   \(n,k\) be two integers such that  \(n \geq 2\) and \(n-1\geq k\geq 1\). Given  a convex polygon with \(n+1\)  labelled edges, the number of  its   dissections of type \((i_1^{m_1},i_2^{m_2},...,i_s^{m_s})\) by  \(k-1\) diagonals  is
\begin{equation}
\label{prescribedformula}  \frac{(n+k-1)!}{n! \ m_1!m_2!\cdots m_s!}  
\end{equation}
\end{cor}

\begin{proof}
Let us consider a convex polygon with \(n+1\)   edges labelled counterclockwise from \(0\) to \(n\), as in the example of Figure \ref{dissection}. As we know,  a dissection of this polygon  by  \(k-1\) diagonals corresponds to a nested set of the list \(1,2,...,n\) and therefore,  in view of Theorem \ref{orderedbijectionnested}, to an admissible internally ordered \(k\)-partition of \([1,n+k-1]\).
By  construction of the bijection of Theorem \ref{orderedbijectionnested}, each  part of this partition describes one of the   polygons  of  the dissection, and  if it has cardinality \(a\) this  polygon  has \(a+1\) edges. 
Therefore  each dissection  of type \((i_1^{m_1},i_2^{m_2},...,i_s^{m_s})\) corresponds to an  admissible   internally ordered partition of \([1,n+k-1]\)  that has \(m_j\) parts of cardinality \(i_j-1\) for every \(j=1,...,s\). 

Let us then  denote by \(A\) the set  of  all the admissible   internally ordered partitions of \([1,n+k-1]\)  that have  \(m_j\) parts of cardinality \(i_j-1\), for every \(j=1,...,s\). 
A quick and elementary computation shows that 
\[|A|=\frac{(n+k-1)!}{m_1!m_2!\cdots m_s!}\]
Now by  Theorem \ref{orderedbijectionnested} we know that each of the partitions in \(A\)  corresponds to a nested set  of a  list that is a  permutation of  the list \(1,2,...,n\).
Then,  by  \(S_n\)-symmetry,  the partitions in \(A\) that correspond to a nested set of the list \(1,2,...,n\), i.e., to a dissection of the prescribed type,  are  exactly 
\[\frac{1}{n!}|A|=\frac{1}{n!}\frac{(n+k-1)!}{m_1!m_2!\cdots m_s!}\]

\end{proof}

Even if our proof  of Kirkman-Cayley  formula 
is    purely combinatorial, here we sketch out   a  geometric interpretation.
 Let us  consider  the real moduli space \({\overline M}_{0,n+1}\) of stable \(n+1\)-pointed  genus 0 curves.

In \cite{gaimrn0}, \cite{gaimrn} some spherical models of subspace arrangements are described, in the spirit of De Concini and Procesi construction of wonderful models in \(\cite{DCP1}\). These spherical models are manifolds with corners, and the minimal spherical model associated with   the root arrangement of type \(A_{n-1}\)
is made by  the disjoint union of \(n!\) copies of the \((n-2)\)-dimensional Stasheff's associahedron (for a concrete realization see \cite{gaiffipermutonestoedra}). Furthermore, there is a surjective map \(\Gamma\) from this  minimal spherical model to the minimal real compact De Concini-Procesi model of type \(A_{n-1}\), that is isomorphic to \({\overline M}_{0,n+1}\). 

This  map \(\Gamma\) sends the \(k-1\)-codimensional faces of the associahedra into the \(k-1\)-codimensional strata of  the boundary of \({\overline M}_{0,n+1}\). We notice that these strata    are indexed by the elements of \(\Sc_2(n,k)\) and the  resulting tessellation coincides with the one previously  described in \cite{Kapranov},  \cite{devadosstessellation} and  \cite{devadossread}.

The  picture above  leads to our computation since  one observes that the  \(k-1\)-codimensional faces in the minimal spherical model are indexed by \(\Oc\Sc_2((1,2,...,n),k)\). 
On one hand they are \(n! D_{n+1,k-1}\), given that  \(D_{n+1,k-1}\) counts the \(k-1\)-codimensional faces of a \((n-2)\)-dimensional Stasheff's associahedron.
On the other hand,   one can count them by regrouping the ones whose images via \(\Gamma\) lie in the  same isomorphism class of boundary strata of \({\overline M}_{0,n+1}\). Now,  two strata belong to the same isomorphism class if and only if their associated nested sets give rise, under the bijection of Theorem \ref{bijectionnested}, to two  partitions whose  parts  have  the same sizes. 
This remark  points out    the bijection, shown in  Theorem \ref{orderedbijectionnested}, between \(\Oc\Sc_2((1,2,...,n),k)\)  and  \(\Ic\Tc_2(n+k-1,k)\).  Then the Kirkman-Cayley formula needs only a last step,  i.e. the computation of the cardinality of \(\Ic\Tc_2(n+k-1,k)\),  that is  provided by the bijection of Theorem \ref{teo:semiordinate}. 


\addcontentsline{toc}{section}{References}
\bibliographystyle{acm}
\bibliography{Bibliogpre} 
\end{document}